\documentclass[leqno,12pt]{article}

\usepackage{amsmath}
\usepackage{amscd}
\usepackage{amsopn}
\usepackage{amsthm}
\usepackage{amsfonts,amssymb}
\usepackage{amsfonts,bbm}
\usepackage{latexsym}

\usepackage{color}

\makeatletter
\@addtoreset{equation}{section}
\makeatother

\newtheorem{lemma}{LEMMA}[section]
\newtheorem{proposition}[lemma]{PROPOSITION}
\newtheorem{corollary}[lemma]{COROLLARY}
\newtheorem{theorem}[lemma]{THEOREM}
\newtheorem{remark}[lemma]{REMARK}

\newcommand{\real}{\mathbbm{R}}
\newcommand{\nat}{\mathbbm{N}}

\newcommand{\limn}{\lim_{n \to \infty}}

\renewcommand{\a}{\alpha}

\newcommand{\g}{\gamma}

\newcommand{\vp}{\varphi}

\newcommand{\reald}{{\real^d}}

\newcommand{\on}{\quad\text{ on }}
\newcommand{\und}{\quad\mbox{ and }\quad}

\newcommand{\ov}{\overline}

\newcommand{\V}{\mathcal V}  
\newcommand{\W}{\mathcal W}

\newcommand{\C}{\mathcal C}  
\newcommand{\E}{{\mathcal E}}
\newcommand{\F}{\mathcal F}

\renewcommand{\H}{{\mathcal H}}
\newcommand{\B}{\mathcal B}

\newcommand{\M}{\mathcal M}

\newcommand{\U}{{\mathcal U}}

\newcommand{\supp}{\operatorname*{supp}}

\newcommand{\itemframe}%
{\setlength{\parskip}{10pt}\begin{enumerate} \setlength{\topsep}{10pt}%
\setlength{\itemsep}{15pt}\setlength{\parsep}{5pt}}

\newcommand{\schluss}{\end{frame}\end{document}}

\newcommand{\bbx}{\B_b(X)}

\newcommand{\pub}{\mathcal P_b(U)}

\newcommand{\px}{\mathcal P(X)}

\newcommand{\hu}{\H_1^+(U)}

\date{}

\title{Equicontinuity of harmonic functions\\ and compactness of potential kernels}
\author{Wolfhard Hansen}
\begin{document}

\maketitle

\begin{abstract}
  Within the framework of balayage spaces (the analytical equivalent of nice Hunt processes),
  we prove  equicontinuity of bounded families of harmonic functions and apply it to
  obtain   criteria for compactness of potential kernels.
\end{abstract}

\section{Introduction}

The main purpose of this paper is to provide simple
criteria for compactness of potential kernels
(Proposition \ref{compact} and Corollary \ref{final})
in the general framework of  balayage spaces $(X,\W)$
(specified by $(B_0)$\,--\,$(B_3)$  and Remarks \ref{Hunt}, \ref{Hunt-converse} below).
These results shall be essential in a~forthcoming paper \cite{H-Bogdan-semi}.

They are based on  \cite[Lemma~3.1]{H-modification},
  where compactness of potential kernels  
   for continuous real potentials   with compact superharmonic
   support   has been stated.  Its proof used
   (local) equicontinuity of bounded families of harmonic functions without providing   any details or   references.
   Therefore we shall first prove such an equicontinuity before getting to compactness
   of potential kernels.

In the following let $X$ be a locally compact space with countable base. For every open set~$U$ in~$X$,
let~$\B(U)$ ($\C(U)$ resp.)  denote the set of all Borel measurable numerical functions  (continuous
real functions resp.) on $U$.
Further, let $\C_0(U)$   be the set of all    functions in $\C(U)$ which vanish at infinity
with respect to $U$.
Given any set $\F$ of functions, let~$\F_b$ ($\F^+$ resp.) denote the set of  
 bounded (positive resp.) functions in $\F$.

We recall that $(X,\W)$ is  called a \emph{balayage space},   if $\W$ is a convex cone of positive numerical functions
on~$X$ such that $(B_0)$ -- $(B_3)$ hold:
\begin{itemize} 
\item  [{\rm (B$_0$)}] 
  $\W$ has the  following continuity, separation and transience properties: 
  \begin{itemize}
\item[\rm (C)]  
Every $w\in \W$ is the supremum of its minorants in $\W\cap\C(X)$.
\item[\rm (S)]  
  For all $x\ne y$ and  $\g>0$, there is a~function $v\in \W$ 
such that  $v(x)\ne \g v(y)$.
\item[\rm (T)]  
There are strictly positive functions $u,v\in\W\cap \C(X)$  with~$u/v\in\C_0(X)$.
  \end{itemize} 
\item   [{\rm (B$_1$)}] 
If $v_n\in \W$, $v_n\uparrow v$, then $v\in \W$.
\item  [{\rm (B$_2$)}] 
If $\V\subset \W$, then $\widehat{\inf \V}^f\in\W$. 
\item   [{\rm (B$_3$)}] 
If $u,v',v''\in\W$,  $u\le v'+v''$, then  there exist $u',u''\in\W$ such that  $u=u'+u''$ and $u'\le v'$,  $u''\le v''$. 
\end{itemize}
 Here $\hat g^f$ is the greatest finely lower semicontinuous minorant of $g$, where the ($\W$-)\emph{fine topology}
on $X$  is the coarsest topology 
 such that  functions in $\W$ are continuous.

 \begin{remark}\label{Hunt}
   {\rm
 If
     $\mathbbm P=(P_t)_{t>0}$ is a sub-Markov semigroup on $X$ (for example, the transition semigroup
      of a Hunt process  $\mathfrak X$) such that its convex cone 
\begin{equation*} 
\E_{\mathbbm P} :=\{u\in \B^+(X)\colon \sup_{t>0} P_tu=u\}
\end{equation*} 
of excessive functions satisfies $(B_0)$, then $(X,\E_{\mathbbm P})$ is a balayage space;
see \cite[II.4.9]{BH} or  \cite[Corollary 2.3.8]{H-course}. We might note that the essential part of $(B_0)$, the continuity
property $(C)$,   holds, if the resolvent kernels $V_\lambda:=\int_0^\infty P_t\,dt$, $\lambda>0$, are strong
Feller, that is, $V_\lambda(\B_b(X))\subset  \C_b(X)$. --
For a converse, see Remark \ref{Hunt-converse}.
}
\end{remark}

For this and  an exposition of  the theory of balayage spaces in   detail,
see \cite{BH,H-course}; for      %some basic information
a description,  which is   more expanded than the one given here and includes
a~discussion of examples,
we mention \cite{H-prag-87} and  \cite[Appendix 8.1]{HN-unavoidable-hameasures}.

In the following, let $(X,\W)$ be a balayage space.  The reader may have in mind  that
      we mostly can assume without loss of generality that $1\in \W$, since
      $(X, (1/s_0)\W) $ is a  balayage space       for every  strictly positive  $s_0\in \W\cap \C(X)$.
    
We recall that the set $\px$ of \emph{continuous real potentials on $X$}   is defined by
\begin{equation*}
      \px:=\{p\in\W\cap \C(X)\colon \exists\  w\in \W\cap \C(X), \, w>0,
  \,  \hbox{with } p/w\in\C_0(X)\}.   
\end{equation*}
Of course, $\px$ is a convex cone.
% If $w_1,w_2\in \W$ such that $w_1+w_2\in \px$, -then clearly $w_1,w_2\in \px$.
Moreover, every function in $\W$ is the limit of an increasing sequence
in $\px$ and, for every $p\in \px$, there exists a strictly positive $q\in \px$ such that $p/q\in \C_0(X)$
(see \cite[II.4.6]{BH} or \cite[Proposition 1.2.1]{H-course}).

A general minimum principle 
implies that, for every $p\in\px$, there exists a~smallest
closed set $C(p)$  in $X$ (it is the closure of the \emph{Choquet boundary} of the function cone $\px+\real p$) such that
\begin{equation}\label{dom}
p=  R_p^{C(p)}:=\inf\{w\in \W\colon w\ge p\mbox { on } C(p)\},
\end{equation}
called  \emph{carrier} or \emph{superharmonic support}  of $p$ (see \cite[II.6.3]{BH} or \cite[Proposition 4.1.6]{H-course}).   
For every $p\in \px$, there is a~unique kernel $K_p$ on $X$, called \emph{associated potential kernel}
or \emph{potential kernel for $p$}, such that the following holds (see \cite[II.6.17]{BH}):
\begin{itemize}
\item
  $K_p1=p$.
\item
   For every $f\in \B_b^+(X)$,    $K_pf\in \px$ and $C(K_pf)\subset \supp(f)$.
 \end{itemize}
 
Further, 
 we recall that, for every open set~$V$ in~$X$,
we have a~\emph{harmonic kernel}~$H_V$ given by
\begin{equation}\label{HV}
  H_Vp=R_p^{X\setminus V}= \inf\{ w\in \W\colon w\ge p\mbox{ on } X\setminus V\}
\end{equation}
for every $p\in \px$ (see \cite[p.\,98 and II.5.4]{BH} or \cite[Section 4.2]{H-course}). 
 
 \begin{remark}\label{Hunt-converse}
      {\rm
 The following holds (see \cite[II.8.6, proof of  IV.8.1 and VI.3.14]{BH}):  If   $1\in \W$, then, for every  $p\in \mathcal P_b(X)$ which is strict
  (that is, satisfies $K_p1_W\ne 0$ for every finely open Borel set $W\ne \emptyset$),
  there exists a  Hunt processes $\mathfrak X$ on $X$ such that its transition semigroup $\mathbbm P=(P_t)_{t>0}$ satisfies
  \begin{equation*}
    \E_{\mathbbm P}=\W \und \int_0^\infty P_t\,dt =K_p
  \end{equation*}
  (so that, in particular, the resolvent kernels are strong Feller).
  
  If $\tau_V$ is the \emph{exit time} of an open set $V$,  that is,
$\tau_V:=\inf\{t\ge 0\colon X_t\notin V\}$, then
\begin{equation*}
\mathbbm E^x(f\circ X_{\tau_V})=  H_Vf(x), \qquad f\in \B^+(X),\,x\in X.
\end{equation*}  
}
\end{remark}

Now let $U$ be an open set in $X$ and let $\U(U)$ denote the set of all open sets~$V$ with compact closure in $U$.
As  usual, let ${}^\ast\H(U)$ denote the set of  functions $u\in \B(X)$  which are \emph{hyperharmonic on $U$}, that is,
are lower semicontinuous on $U$ and satisfy
\begin{equation*}
  -\infty <H_Vu(x)\le u(x) \quad\mbox{ for all } x\in V\in \U(U).
\end{equation*}
We note that ${}^\ast\H^+(X)=\W$ (see \cite[II.5.5]{BH} or \cite[Proposition 4.1.7]{H-course}).  
The set $\H(U):={}^\ast\H(U)\cap (-{}^\ast\H(U))$  is the set of   functions in $\B(X)$ which are \emph{harmonic on~$U$},
that is, 
\begin{equation*}
  \H(U) =\{h\in\B(X)\colon h|_U\in \C(U), \ H_Vh=h\mbox{ for every }V\in \U(U)\}.
\end{equation*}

The equalities  (\ref{dom}) and (\ref{HV} immediately imply that 
the superharmonic support $C(p)$ for $p\in\px$ is the smallest closed set such that $p$ is harmonic on its complement.

In particular, we have  the following, if  there is a Green function~$G $ for $(X,\W)$. If $p\in \px$
such that
    $ 
      p=G \mu:=\int G (\cdot,y)\,d\mu(y)
      $
      for some measure $\mu\ge 0$ on $X$
    (see~\cite{HN-representation} for such a  representation), then, for every $f \in \B^+(X)$,
    \begin{equation*}
      K_pf=G (f\mu). 
    \end{equation*}

    By \cite[III.2.8 and III.1.2]{BH},
    \begin{equation}\label{char-px}
        \px=\{p\in \W\cap C(X)\colon \mbox{If $h\in \H^+(X)$ and $h\le p$, then $h=0$}\},
        \end{equation} 
        and
        \begin{equation}\label{ex-harmonic}
          H_Uf\in  \H(U), \quad\mbox{ if } f\in \B(X)\mbox{ and } |f|\le s\in \W\cap \C(X). 
        \end{equation}

        Finally, we note that a function $h\in \B^+(X)$ satisfying $h|_U\in \C(U)$ is harmonic on $U$ provided that,
        for every $x\in U$,
        there is a fundamental system $\V(x)\subset \U(U)$ of neighborhoods  of $x$   with  $H_Vh(x)=h(x)$
        for every $V\in \V(x)$
        (see \cite[III.4.4]{BH} or \cite[Corollary~5.2.8]{H-course}). Analogously  for hyperharmonic functions.
 In particular,  for positive functions, being harmonic (hyperharmonic resp.) on an open set is a~local property
in the following sense: If $(U_i)_{i\in I}$ is a family of open sets in $X$, then 
\begin{equation}\label{harmonic-union}
  \bigcap_{i\in I} \H^+(U_i)=            \H^+ (\bigcup_{i\in I} U_i ) \und
  \bigcap_{i\in I} {}^\ast\H^+(U_i)=         {}^\ast   \H^+ (\bigcup_{i\in I} U_i ).
\end{equation}

\section{Equicontinuity of sets of harmonic functions}\label{equicontinuous}

Let $U$ be an open set in $X$,  $s\in \W\cap\C(X)$ and 
\begin{equation*}
  \H_s(U):=\{h\in \H(U)\colon  | h|\le s\}.
  \end{equation*} 
The main result of this section is the following.
  
\begin{theorem}\label{equi}
 The set $\H_s(U)$ is {\rm(}locally{\rm)} equicontinuous on $U$.
\end{theorem}

It will be sufficient to prove that $\H_s^+(U)$ is locally equicontinuous on $U$,
since, for all $V\in \U(U)$ and $h\in \H_s(U)$, we know that $h=H_Vh=H_Vh^+-H_Vh^-$,
where $H_Vh^\pm\in \H_s^+(V)$; see (\ref{ex-harmonic}).

We   assume (without loss of generality) that $s=1$ and 
 establish  
the equicontinuity first  at points of $U$ which are contained in  the set 
 \begin{equation*}
   X_0:=\{x\in X\colon \lim_{V\downarrow x} H_V(x,W)=1
   \mbox{ for every open neighborhood $W$ of $x$}\}.
 \end{equation*}
In many cases,  for example for harmonic spaces and for the balayage space given by Riesz potentials
(symmetric $\a$-stable processes) on $\reald$, we have $X_0=X$. 
 
Our approach is based on (\ref{ex-harmonic}) 
and inspired by  a   result in~\cite{meyer-moko} on the composition
 of two strong Feller kernels.  We start with the following lemma.

    \begin{lemma}\label{pointwise-on-V}
   Let $V$ be a relatively compact open set in $X$ and let $(f_n)$ be a~sequence in $\B(X)$,
   $0\le f_n\le 1$ for every $n\in\nat$. Then there exists a subsequence~$(f_n')$ of~$(f_n)$
   and a function $f\in \B(X)$ such that $0\le f\le 1$ and the sequence $(H_Vf_n')$ converges
   pointwise to $H_Vf$ on $V$.
 \end{lemma}

 \begin{proof} Let $\{x_m\colon m\in \nat\}$ be a dense sequence in $V$ and
   $\sigma:=\sum_{m=1}^\infty  2^{-m}H_V(x_m,\cdot)$. Then $\sigma(X)\le 1$. 
   Since $L^{\infty}(\sigma)$ is the dual of $L^1(\sigma)$, there exists a~subsequence $(f_n')$
   of~$(f_n)$ and a function $f\in\B(X)$  such that $0\le f\le 1$ and
   \begin{equation}\label{dual}
     \limn \int f_n' g\,d\sigma=\int fg\,d\sigma\qquad\mbox{ for every }g\in \mathcal L^1(\sigma).
   \end{equation}
   Let $x\in V$ and let $A$ be a Borel set in $V$ such that $\sigma(A)=0$. Then $H_V1_A(x_m)=0$
   for every $m\in\nat$, and hence $H_V(x,A)=0$, since 
   the function    $H_V1_A$ is continuous on $V$, by (\ref{ex-harmonic}).
   So, by the theorem of Radon-Nikodym, there exists   $g\in \mathcal L^1(\sigma)$
   such that $H_V(x,\cdot) = g\sigma$. By (\ref{dual}), we conclude that $\limn H_Vf_n'(x)=H_Vf(x)$.
  \end{proof} 

  \begin{corollary}\label{pointwise}
    Every sequence $(h_n)$ in $\H_1^+(U)$ has a~subsequence~$(h_n')$ which converges pointwise
    on~$U$.
  \end{corollary}

  \begin{proof} If $V$ is an open set in $U$ and $\ov V$ is compact in $U$, then,
    by Proposition \ref{pointwise-on-V}, there exists a subsequence $(h_n')$ of $(h_n)$ such that
    $(H_V h_n')$ converges pointwise on $V$. This means that $(h_n')$ converges pointwise on $V$,
    since $H_Vh_n=h_n$ for every $n\in\nat$. The proof is completed choosing a covering
    of $U$ by a sequence of such sets~$V$ and using a standard diagonal procedure.
  \end{proof}

  \begin{proposition}\label{x0-result}
 If $x\in U\cap X_0$, then  $\H_s^+(U)$ is equicontinuous at $x$. 
\end{proposition}

\begin{proof} 
Let us suppose that $\hu$ is not equicontinuous at a point $x\in U\cap X_0$. We have to show that this
    leads to a contradiction. To that end 
   let $(A_n)$ be a sequence of compact neighborhoods of $x$
  in $U$ such that $A_n\downarrow \{x\}$ and $A_{n+1}$ is contained in the interior of $A_n$, $n\in\nat$. 
 Then there exists $\delta\in (0,1)$ such that,
  for every $n\in\nat$, there are $h_n\in \hu$ and $y_n\in A_n$ satisfying
  \begin{equation}\label{contra}
    |h_n(y_n)-h_n(x)|\ge 5 \delta.
  \end{equation}
  By Corollary \ref{pointwise}, we may assume that the sequence $(h_n)$ converges pointwise  on $U$.
 
  Since $x\in X_0$, there exists an open neighborhood $V$ of $x$
  such that $\ov V$ is compact in~$U$ and $H_V1_U (x) >1-\delta$.
  By continuity of  
  $H_V1_U $ on $V$, there exists $n_0\in\nat$ such that 
  $A:=A_{n_0}\subset V$ and
    $H_V1_U >1-\delta$ on $A$. Since $H_V1\le 1$, we obtain that 
  \begin{equation}\label{delta-est}
    H_V  1_{X\setminus U} <\delta \on A.
  \end{equation}

  We now define
    \begin{equation*}
      g_n:=1_U  h_n \mbox{ for every }n\in\nat \und  g:=\limn g_n.
      \end{equation*} 
  By (\ref{delta-est}), for every $n\in \nat$, 
\begin{equation}\label{hhg}
      |h_n-H_Vg_n|=|H_V(h_n-g_n)|\le H_V1_{X\setminus U} <\delta \on A.
    \end{equation}
     For every $n\in\nat$, let
    \begin{equation*} 
     g_n':=\inf_{k\ge n} g_k \und g_n'':=\sup_{k\ge n} g_k .
      \end{equation*} 
   $H_Vg_n'\uparrow H_V g$ and $H_Vg_n''\downarrow H_Vg$ pointwise as $n\to \infty$.
    Since these functions are continuous on $V$,  the convergence is uniform on $A$.
    Of course, $H_Vg_n'\le H_Vg_n\le H_Vg_n''$ for every $n\in\nat$.  So the sequence
    $(H_Vg_n)$ converges to $H_Vg$  uniformly on $A$. Hence there exists $n_1\ge n_0$ such that,
for every $n\ge n_1$, 
    \begin{equation}\label{hgg}
      |H_Vg_n-H_Vg|<\delta \on A. 
    \end{equation}
    Further,  by continuity of $H_Vg$ on $V$, there exists $n\ge n_1$ such that
\begin{equation}\label{hghg}
      |H_Vg-H_Vg(x)|<\delta \on A_n.
    \end{equation}
    
   Finally,   combining the estimates (\ref{hhg}),  (\ref{hgg})  and (\ref{hghg}) we obtain that
      \begin{equation*}
        |h_n-h_n(x)|< 5\delta \on A_n
      \end{equation*}
      contradicting (\ref{contra}). Thus $\hu$ is equicontinuous at $x$.
\end{proof}

To continue our proof of Theorem \ref{equi} (and for later use) we define
\begin{equation*}
       \W_U:= {}^\ast\H^+(U)|_U
     \end{equation*}
     and observe 
     that $(U,\W_U)$ is a balayage space 
     (see \cite[V.1.1]{BH}). 
     
     Let us now consider a point  $x\in X\setminus X_0$.
     By \cite[III.2.7]{BH}, it  is finely isolated. 
Since $\W|_U \subset \W_U$, 
it is also finely isolated  with respect to $(U,\W_U)$. Therefore 
\begin{equation*}
  q_x:=\inf\{w\in \W_U\colon w(x)\ge 1\}
  \end{equation*} 
  is a continuous real potential for $(U,\W_U)$ with $C(q_x)=\{x\}$  
  (see  \cite[p.\,94 and III.2.8]{BH} or ~\cite[Lemma 4.2.13]{H-course}).

 \begin{lemma}\label{qx} The set $\hu$ is equicontinuous at every point $x\in U\setminus X_0$.
 \end{lemma}

 \begin{proof} 
   Given $\delta>0$,  there exists a neighborhood $V$ of $x$ in $U$ such that, for all $y\in V$,
   \begin{equation}\label{qx-continuous}
    q_x >1-\delta   \on V.  
   \end{equation}
   We  now fix $h\in \hu$. Then the restrictions $v,w$ of $h, 1-h$, respectively,  on $U$ are hyperharmonic
   on $U$. 
   Applying (\ref{dom})
  to the balayage space $(U,\W_U)$ we get~that
   \begin{equation*}
     v\ge h(x) q_x \und w\ge (1-h(x))q_x.
   \end{equation*}
 Since $0\le h(x)\le 1$, this implies that,  for every $y\in V$, by (\ref{qx-continuous}),    
   \begin{equation*}
     h(y)> h(x) - \delta  \und
   1-h(y)> 1-h(x)- \delta, 
   \end{equation*}
  that is,  $\delta >h(x)-h(y)>- \delta$. 
   \end{proof} 

   Having Proposition \ref{x0-result} and Lemma \ref{qx}  the proof of Theorem \ref{equi} is completed.

    \begin{remark} {\rm
        For  the sake of completeness let us finally note that $X\setminus X_0$ is the set of \emph{all } finely
        isolated points in $X$ and that it is (at most) countable. If $X$ is a (countable) discrete space,
        then, of course, $X_0=\emptyset$. In \cite{H-limits} it is shown that, for   $X=(0,1)$, the set
        $X\setminus X_0$ can be \emph{any} given countable subset of $X$.
        }
        \end{remark}

\section{Equicontinuity of specific minorants of   $p\in \px$}\label{specific}

Let $\prec$ denote the \emph{specific order} on $\W$, that is, if $u,v\in \W$,
then $u\prec v$ if there exists $w\in \W$ such that $u+w=v$. If $q\in \px$
and $f\in \B_b(X)$ such that $0\le f\le 1$, then $K_qf\prec q$, since
$K_q(1-f)\in \px$. If $q,q' \in\px$,
then obviously $K_{q+q' }=K_{q}+K_{q' }$, and hence $K_{q}f\prec K_{q+q' }f$
  for every $f\in \B^+(X)$. For $q\in \px$ and Borel sets~$A$ in~$X$, let
  \begin{equation*}
    q_A:=K_q1_A.
    \end{equation*}

    Having Theorem \ref{equi} the proof given in \cite{H-quasi-balayage} for the following
    result  is complete. However, for the convenience of  the reader we add a quick presentation.

  \begin{proposition}\label{spec}
    For every $p\in \px$, the set $\M_p:=\{q\in \px\colon q\prec p\}$
    is {\rm(}locally{\rm)} equicontinuous on $X$.
     \end{proposition}

     \begin{proof} Let $x\in X$ and $\delta>0$. There exists an open neighborhood~$U$
       of $x$ such that $p_{U\setminus \{x\}}(x)<\delta$,  and hence $p_{U\setminus \{x\}}<\delta$
         on some neighborhood $V$ of $x$. Moreover, we may assume that $ 
         |p_{\{x\}}-p_{\{x\}}(x)|<\delta$ on $V$ (if $\{x\}$ is totally thin, then $p_{\{x\}}=0$).  
         By Proposition \ref{equi}, there exists a~neighborhood~$W$
         of $x$ in $V$ such that, for every $q\in \M_p$, $|q_{X\setminus U}-q_{X\setminus U}(x)|<\delta$ on $W$.
        
         Now let us fix $q\in \M_p$. Then $q_{U\setminus \{x\}}\prec  p_{U\setminus \{x\}}$
        and  $ q_{\{x\}}\prec  p_{\{x\}}$. By (\ref{dom}), $ q_{\{x\}}=\a  p_{\{x\}}$ with $\a\in[0,1]$. Thus
        $|q-q(x)| < 3\delta$ on $W$.
  \end{proof}
  
  \begin{corollary}\label{Kp-equi}
    For every $p\in\px$, the set $\{K_pf\colon f\in \B(X), \, 0\le f\le 1\}$ is
    {\rm(}locally{\rm)} equicontinuous on $X$.
  \end{corollary}

  At first sight, Proposition \ref{spec} may look stronger than Corollary \ref{Kp-equi}.
  However, it is not, since, for every $q\prec p$, there exists a function $f\in\B(X)$
  such that $0\le f\le 1$ and $K_pf=q$; see \cite[II.7.11]{BH}.

  \section{Compactness of potential kernels}\label{criteria}

  Let us introduce the following boundedness property for $(X,\W)$ (cf.\  Remark \ref{compact-general}):
  \begin{itemize}
  \item[\rm (B)] There is a strictly positive bounded function $w_0\in \W$.
\end{itemize}

We  first recall the statement of \cite[Lemma 3.1]{H-modification} and prove it using Corollary \ref{Kp-equi}.

\begin{proposition}\label{KC-compact} Suppose {\rm (B)} 
  and let  $p\in\px$ such that  $C(p)$ is compact. Then $K_p$ is a~compact operator on~$(B_b(X), \|\cdot\|_\infty)$.    
\end{proposition}

\begin{proof} Let $(f_n)$ be a bounded sequence in $B_b(X)$, without loss of generality $0\le f_n\le 1$
  for every $n\in\nat$. By Corollary \ref{Kp-equi} and the theorem of Arzel\`a-Ascoli,
  there exists a subsequence $(q_n)$ of $(Kf_n)$ which is uniformly convergent on $C(p)$.

  Let  $\delta>0$, $a:=\inf w_0(C(p))$ and  $b:=\sup w_0(C(p))$. There exists $k\in\nat$
  such that, for all $m,n\ge k$,
  \begin{equation*}
    q_m < q_n+ (\delta/b) a 
    \on C(p), 
  \end{equation*}
  where $C(q_m)\subset C(p)$, and  therefore  $q_m\le q_n+(\delta/b)w_0\le q_n+\delta$ on $X$, by   (\ref{dom}).
  So~the sequence $(q_n)$ is uniformly convergent.
  \end{proof} 

  Given $g\in \B_b(X)$, let us denote the operator $f\mapsto fg$ on $B_b(X)$ by $M_g$.
Clearly, for all $g\in \B_b^+(X)$ and  $p\in\px$,  the potential kernel for $K_pg$ is  $K_pM_g$.

\begin{corollary}\label{ex-vp}
  Suppose {\rm (B)} 
    and let  $p\in\px$. Then there exists  a~function  $\vp_0\in \C(X)$, $0<\vp_0\le 1$, such that  the potential kernel of
    $K_p\vp_0$ is a~compact operator on~$(B_b(X), \|\cdot\|_\infty)$.
 \end{corollary}  

 \begin{proof} Let  us choose $\vp_n\in \C(X)$ with compact support, $0\le \vp_n\le 1$, such that
   $\bigcup_{n\in\nat} \{\vp_n>0\}=X$. 
   For every $n\in\nat$,  $p_n:=K_p\vp_n\in \mathcal P(X)$ with  $C(p_n)\subset \supp(\vp_n)$, and hence 
   $K_{p_n}$ is a compact operator on $(\bbx,\|\cdot\|_\infty)$, by Theorem \ref{KC-compact}. 
   Let $0<\a_n\le 2^{-n}$, $n\in\nat$,  such that $\a_n p_n\le 2^{-n}$.
   Then $\vp_0:=\sum_{n=1}^\infty \a_n\vp_n\in\C(X)$,
$0<\vp_0\le 1$, $p_0:=K_p\vp_0=\sum_{n=1}^\infty \a_n p_n\in \mathcal P_b(X)$ and
$K_{p_0}=\sum_{n=1}^\infty \a_n K_{p_n}$ is a~compact operator on~$(\bbx, \|\cdot\|_\infty)$.
    \end{proof}

Fixing an exhaustion   of $X$ by relatively compact open sets $U_n$, $n\in\nat$,  we have the following.

 \begin{theorem}\label{compact}
   
   Assuming {\rm (B)} the following  
   are equivalent for every $p\in \px$:
\begin{itemize}
\item[\rm (1)]
$K_p$ is a~compact operator on  $(\B_b(X),\|\cdot\|_\infty)$  
{\rm(}and  $K_p(\B_b(X)) \subset \C_b(X)${\rm)}. 
\item[\rm (2)]  
  $\limn \|K_p1_{X\setminus U_n}\|_\infty=0$.\footnote{If $1\in\W$, then
    $\|K_p1_{X\setminus U_n}\|_\infty=\sup\{ K_p1_{X\setminus U_n}(x)\colon x\in X\setminus U_n\}$,
    by (\ref{dom}).} 
\end{itemize} 
\end{theorem} 

\begin{proof}
  (1)\,$\Rightarrow$\,(2) 
  Almost trivial: Obviously, $K_p1_{X\setminus U_n}\downarrow 0$ pointwise as $n\to\infty$.
  By~compactness of $K_p$, the sequence $(K_p1_{X\setminus U_n})$
contains a uniformly convergent subsequence. Thus  (2) holds.

(2)\,$\Rightarrow$\,(1)
For every $n\in\nat$, $p_n:=K_p1_{U_n}\in \mathcal P(X)$ and $C(p_n)\subset \ov U_n$.
Hence $K_{p_n}$ is a~compact operator on $\B_b(X)$, by Proposition \ref{KC-compact},
and $K_{p_n}(\B_b(X))\subset \C_b(X)$.
Since
\begin{equation*}
         K_p=K_pM_{1_{U_n}}+K_pM_{1_{X\setminus U_n}}, \qquad n\in\nat,
       \end{equation*}
where $K_pM_{1_{U_n}}=K_{p_n}$ and $\|K_pM_{1_{X\setminus U_n}}\|_\infty=\|K_p1_{X\setminus U_n}\|_\infty$,
we obtain that $K_p$ is a~compact operator on~$\B_b(X)$ and
 $K_p(\B_b(X))\subset C_b(X)$. 
\end{proof}

\begin{remark}\label{compact-general} {\rm
    To apply Theorem \ref{compact} 
    in the general case, we~may choose $s_0\in\W\cap \C(X)$, $s_0>0$, and consider the balayage space
    $(X,\widetilde \W)$ with~$\widetilde \W:=(1/s_0) \W$ and $1\in \widetilde \W$.
    Given $p\in \px$, the function $\tilde p:=p/s_0$ is a~continuous real
    potential for $(X,\widetilde \W)$ and the associated potential kernel $\widetilde K_{\tilde p}$ is given by
    \begin{equation*}
      \widetilde K_{\tilde p} f=(1/s_0) K_p f, \qquad f\in \B^+(X).
    \end{equation*}
    Then compactness of $\widetilde K_{\tilde p}$ on $(\B_b(X), \|\cdot\|_\infty)$ implies compactness
    of~$K_p$ on the space of all $s_0$-bounded functions equipped with the norm
    \begin{equation*}
      \|f\|:=\inf \{a\ge 0\colon |f|\le a s_0\}.
    \end{equation*}
  }
  \end{remark}

  In the following let $U$ be an open set in $X$. We recall from the preceding section
  that taking $\W_U:= {}^\ast\H^+(U)|_U$ we obtain a balayage space $(U,\W_U)$.
  By  \cite[V.1.1]{BH}, we know, in addition,  that   $ q-H_Uq\in \mathcal P(U)$ for every $q\in\px$.

\begin{corollary}\label{final}
Let $U$ be relatively compact. Then the  following hold:
  \begin{itemize}
    \item[\rm (a)]
    If $q\in \px$ and $p:=q-H_Uq$, then $K_p=K_q-H_UK_q$, and $K_p$  is a compact operator on $\B_b(U)$. 
  \item[\rm (b)] If $U$ is regular, then $K_p$ is a compact operator on $\B_b(U)$
    if and only if $p\in \C_0(U)$ {\rm(}and $K_p(\B_b(U))\subset \C_0(U)${\rm)}.
    \end{itemize} 
    \end{corollary}

    \begin{proof} 
     If $w\in \W\cap \C(X)$, $w>0$, then $w_0:=w|_U\in \W_U\cap \B_b(U)$. 

     (a) Obviously,  $K_p=K_q-H_UK_q$. Let $(V_n)$ be an exhaustion of $U$. 
     We have $q_n:=K_q1_{U\setminus V_n}\downarrow 0$, where
      $ q_n\in \C(X)$, and hence $q_n\downarrow 0$ uniformly on $\ov U$ as $n\to\infty$.
      Since $  K_p1_{U\setminus V_n}\le q_n$ for every $n\in\nat$, we~obtain that
      $\limn\| K_p1_{U\setminus V_n}\|_\infty=0$. Thus~$K_p$ is a~compact operator
      on~$\B_b(U)$, by~Theorem \ref{compact}.

      (b) By Theorem \ref{compact}, it suffices to observe that $p_n:=K_U1_{V_n}\in\C_0(U)$  for every $n\in\nat$.
      (To see this we may take any strict potential $q_0\in \px$, consider $p_0:=q_0-H_Uq_0\in \pub\cap \C_0(U)$,
      fix $x\in\nat$, 
      and note that $p_n\le a p_0$ on $\ov V_n$ for some $a>0$, hence $p_n\le a p_0$ on $U$.)
      \end{proof}

{\small \noindent 
Wolfhard Hansen,
Fakult\"at f\"ur Mathematik,
Universit\"at Bielefeld,
33501 Bielefeld, Germany, e-mail:
 hansen$@$math.uni-bielefeld.de}

\end{document}